\documentclass[11pt,reqno]{amsart}
\usepackage{amsmath, amsthm, amssymb, amsfonts}
\usepackage{bbm}
\usepackage{bm}
\usepackage[utf8]{inputenc}
\usepackage{graphicx}
\usepackage{xcolor}
\usepackage{enumitem}
\usepackage{tikz}
\usepackage{tikz-cd}
\usepackage{pgfplots}
\usepackage{rotating}
\usepackage{bookmark}
\usepackage{thmtools}
\usepackage{hyperref}
\usepackage{cleveref}

\pgfplotsset{compat=1.18}

\theoremstyle{plain}

\usepackage[backend=biber,sorting=nty]{biblatex}

\numberwithin{equation}{section}

\newtheorem{theorem}{Theorem}[section]
\newtheorem{lemma}[theorem]{Lemma}
\newtheorem{proposition}[theorem]{Proposition}

\newtheorem{claim}[theorem]{Claim}

\theoremstyle{definition}
\newtheorem{definition}[theorem]{Definition}

\newtheorem*{axiom}{Axiom}

\newcommand{\func}{\operatorname}

\newcommand{\R}{\mathbb{R}}

\newcommand{\N}{\mathbb{N}}

\renewcommand{\bar}[1]{\overline{#1}}

\newcommand{\cal}{\mathcal}

\renewcommand{\bf}{\mathbf}

\newcommand{\ZFC}{\mathsf{ZFC}}

\newcommand{\forces}{\Vdash}
\newcommand{\one}{\mathbbm{1}}
\newcommand{\restrictedto}{\mathord{\upharpoonright}}

\newcommand{\Coll}{\func{Coll}}

\newcommand{\V}{\bf{V}}
\newcommand{\W}{\bf{W}}
\newcommand{\X}{\cal{X}}

\newcommand{\lh}{\func{lh}}

\newcommand{\depth}{\func{depth}}

\newcommand{\fin}{\mathrm{fin}}
\newcommand{\D}{\cal{D}}

\newcommand{\E}{\cal{E}}

\NewCommandCopy\strokeL\L
\DeclareRobustCommand{\L}{\ifmmode\mathbf{L}\else\strokeL\fi}
\NewCommandCopy\pilcrowP\P
\DeclareRobustCommand{\P}{\ifmmode\mathbb{P}\else\pilcrowP\fi}
\NewCommandCopy\accH\H
\DeclareRobustCommand{\H}{\ifmmode\cal{H}\else\expandafter\accH\fi}
\NewCommandCopy\slashO\O
\DeclareRobustCommand{\O}{\ifmmode\cal{O}\else\slashO\fi}
\NewCommandCopy\Sec\S
\DeclareRobustCommand{\S}{\ifmmode\mathcal{S}\else\Sec\fi}

\newenvironment{midproof}[1][Proof]
  {\begin{proof}[#1]}
  {\end{proof}}

\makeatletter
\DeclareRobustCommand\bigop[1]{%
  \mathop{\vphantom{\sum}\mathpalette\bigop@{#1}}\slimits@
}
\newcommand{\bigop@}[2]{%
  \vcenter{%
    \sbox\z@{$#1\sum$}%
    \hbox{\resizebox{\ifx#1\displaystyle.9\fi\dimexpr\ht\z@+\dp\z@}{!}{$\m@th#2$}}%
  }%
}
\makeatother

\DeclareFontShape{OT1}{cmr}{bx}{sc}{<-> cmbcsc10}{}

\title[Every set is Ramsey in the Solovay model]{Every subset of a topological Ramsey space is Ramsey in the Solovay model}
\author{Clement Yung}

\begin{document}
\begin{abstract}
    We provide a proof, free of ultrafilters and almost reduction, that every subset of a topological Ramsey space is Ramsey in the Solovay model.
\end{abstract}

\maketitle

\section{Introduction}
\label{sec:intro}
Let $[\N]^\infty$ denote the set of infinite subsets of the set of natural numbers $\N$. Given an infinite $A \subseteq \N$ and a finite $a \subseteq A$, $[a,A]$ denotes the set of all infinite subsets $B \subseteq A$ such that $a$ is an initial segment of $B$ (i.e. $B \cap (\max(a) + 1) = a$, where $\max(\emptyset) := -1$). A subset $\X \subseteq [\N]^\infty$ is \emph{(completely) Ramsey} if for all $A \in [\N]^\infty$ and finite $a \subseteq A$, there exists some $B \in [a,A]$ such that $[a,B] \subseteq \X$ or $[a,B] \subseteq \X^c$. 

\begin{theorem}[Mathias, \cite{M77}]
\label{thm:Mathias}
    Let $\kappa$ be an inaccessible cardinal, and let $G$ be $\Coll(\omega,{<}\kappa)$-generic. Then in $\L(\R)^{\V[G]}$, every subset of $[\N]^\infty$ is Ramsey.
\end{theorem}

Much of the classical Ramsey theory of $[\N]^\infty$ has been extended to topological Ramsey spaces (i.e. closed triples $(\cal{R},\leq,r)$ satisfying the axioms \textbf{A1}--\textbf{A4} introduced by Todor\v{c}evi\'{c} in \cite{T10}), so it is natural to ask whether Theorem \ref{thm:Mathias} holds in this generality. The key obstacle lies in Mathias' use of selective ultrafilters in his arguments --- while the almost subset relation $\subseteq^*$ serves as a natural $\sigma$-closed quasi-order on $[\N]^\infty$ that allows diagonalisation of sequences, there is no analogue that is $\sigma$-closed for topological Ramsey spaces in general. In \cite{DMN15}, Di Prisco--Mijares--Nieto developed the abstract Mathias forcing $\mathbb{M}_\cal{R}$, and proved that $\mathbb{M}_\cal{R}$ satisfies the Mathias and Prikry properties if the almost reduction relation $\leq^*$ (where $A \leq^*B$ iff $\emptyset \neq [a,A] \subseteq [a,B]$ for some $a \in \cal{AR}$) is $\sigma$-closed. While the $\sigma$-closed hypothesis on $\leq^*$ was removed in the author's paper \cite{Y25} (so $\mathbb{M}_\cal{R}$ satisfies both the Mathias and Prikry properties for every topological Ramsey space), a proof of the corresponding generalisation of Theorem \ref{thm:Mathias} has not appeared in the literature.

In this paper, we review the axioms of topological Ramsey spaces and the abstract Mathias forcing, and prove the following:

\begin{theorem}
\label{thm:main}
    Let $(\cal{R},\leq,r)$ be a closed triple satisfying \textbf{A1}--\textbf{A4} such that $\cal{AR}$ is countable. Let $\kappa$ be an inaccessible cardinal, and let $G$ be $\Coll(\omega,{<}\kappa)$-generic. Then in $\L(\R)^{\V[G]}$, every subset of $\cal{R}$ is Ramsey.
\end{theorem}

\section{Topological Ramsey spaces}
\label{sec:TRS}
The axiomatisation of topological Ramsey spaces was initially studied by Carlson and Simpson in \cite{CS90}. In \cite{T10}, Todor\v{c}evi\'{c} presented the four axioms, labelled \textbf{A1}--\textbf{A4}, which are sufficient for a closed triple $(\cal{R},\leq,r)$ to be a topological Ramsey space. We shall recap the axioms in this section.

Here, $\cal{R}$ is a non-empty set, $\leq$ is a quasi-order on $\cal{R}$, and $r : \cal{R} \times \omega \to \cal{AR}$ is a surjective function. We also define a sequence of maps $r_n : \cal{R} \to \cal{AR}$ by $r_n(A) := r(A,n)$ for all $A \in \cal{R}$. Let $\cal{AR}_n \subseteq \cal{AR}$ be the image of $r_n$ (i.e. $a \in \cal{AR}_n$ iff $a = r_n(A)$ for some $A \in \cal{R}$).

\begin{axiom}[\textbf{A1}, Sequencing]
\hfill
    \begin{enumerate}[label=(\arabic*)]
        \item $r_0(A) = \emptyset$ for all $A \in \cal{R}$.

        \item $A \neq B$ implies $r_n(A) \neq r_n(B)$ for some $n$.

        \item $r_n(A) = r_m(B)$ implies $n = m$ and $r_k(A) = r_k(B)$ for all $k < n$.
    \end{enumerate}
\end{axiom}

For each $a \in \cal{AR}$, let $\lh(a)$ denote the unique $n$ such that $a \in \cal{AR}_n$. By Axiom \textbf{A1}(3), this $n$ is well-defined. Axiom \textbf{A1}(3) also implies that if $a \in \cal{AR}_n$ and $k \leq n$, then $r_k(A) = r_k(B)$ for any $A,B \in \cal{R}$ with $r_n(A) = r_n(B) = a$, so we may unambiguously extend $r_k$ to $\{a \in \cal{AR} : \lh(a) \geq k\}$ by letting $r_k(a)$ denote this common value. For $a,b \in \cal{AR}$, we write $a \sqsubseteq b$ iff $a = r_{\lh(a)}(b)$, and for $B \in \cal{R}$, we write $a \sqsubseteq B$ iff $a = r_{\lh(a)}(B)$.

\begin{axiom}[\textbf{A2}, Finitisation]
    There is a quasi-ordering $\leq_\fin$ on $\cal{AR}$ such that:
    \begin{enumerate}[label=(\arabic*)]
        \item $\{a \in \cal{AR} : a \leq_\fin b\}$ is finite for all $b \in \cal{AR}$.

        \item $B \leq A$ iff $\forall n \exists m [r_n(B) \leq_\fin r_m(A)]$.

        \item $\forall a,b,c \in \cal{AR}[a \sqsubseteq b \wedge b \leq_\fin c \to \exists d \sqsubseteq c[a \leq_\fin d]]$.
    \end{enumerate}
\end{axiom}

We may define the \emph{Ellentuck neighbourhoods} as follows: For any $A \in \cal{R}$, $a \in \cal{AR}$ and $n \in \N$, we let:
\begin{align*}
    [a,A] &:= \{B \in \cal{R} : a \sqsubseteq B \wedge B \leq A\}, \\
    [n,A] &:= [r_n(A),A], \\
    [A] &:= [\emptyset,A], \\
    [a] &:= \{B \in \cal{R} : a \sqsubseteq B\}.
\end{align*}
We also define the depth function $\depth_A$, for $A \in \cal{R}$ and $a \in \cal{AR}$, by:
\begin{align*}
    \depth_A(a) := 
    \begin{cases}
        \min\{n < \omega : a \leq_\fin r_n(A)\}, &\text{if such $n$ exists}, \\
        \infty, &\text{otherwise}.
    \end{cases}
\end{align*}

We also define, for each $A \in \cal{R}$ and $n < \omega$:
\begin{align*}
    \cal{AR}\restrictedto A &:= \{a \in \cal{AR} : \exists m[a \leq_\fin r_m(A)]\}, \\
    \cal{AR}_n\restrictedto A &:= \cal{AR}_n \cap \cal{AR}\restrictedto A.
\end{align*}
If $a \in \cal{AR}\restrictedto A$, then:
\begin{align*}
    \cal{AR}\restrictedto[a,A] &:= \{b \in \cal{AR}\restrictedto A : a \sqsubseteq b\}, \\
    \cal{AR}\restrictedto[a] &:= \{b \in \cal{AR} : a \sqsubseteq b\}, \\
    \cal{AR}\restrictedto a &:= \{b \in \cal{AR} : b \leq_\fin a\},  \\
    \cal{AR}_n\restrictedto[a,A] &:= \cal{AR}_n \cap \cal{AR}\restrictedto[a,A], \\
    \cal{AR}_n\restrictedto[a] &:= \cal{AR}_n \cap \cal{AR}\restrictedto[a], \\
    \cal{AR}_n\restrictedto a &:= \cal{AR}_n \cap \cal{AR}\restrictedto a, \\
    r_n[a,A] &:= \{b \in \cal{AR}\restrictedto[a,A] : \lh(b) = n\}, \\
    r_n[A] &:= r_n[\emptyset,A].
\end{align*}
Note that for all $a \in \cal{AR}$, $\depth_A(a) < \infty$ iff $a \in \cal{AR}\restrictedto A$. 

\begin{axiom}[\textbf{A3}, Amalgamation]
    Let $A \in \cal{R}$.
    \begin{enumerate}[label=(\arabic*)]
        \item For all $a \in \cal{AR}\restrictedto A$ and $B \in [\depth_A(a),A]$, $[a,B] \neq \emptyset$.

        \item If $B \leq A$ and $a \in \cal{AR}\restrictedto B$, then there exists $B' \in [\depth_A(a),A]$ such that $\emptyset \neq [a,B'] \subseteq [a,B]$.
    \end{enumerate}
\end{axiom}

\begin{axiom}[\textbf{A4}, Pigeonhole]
    For all $A \in \cal{R}$, $a \in \cal{AR}\restrictedto A$ and $\O \subseteq \cal{AR}_{\lh(a)+1}$, there exists some $B \in [\depth_A(a),A]$ such that $r_{\lh(a)+1}[a,B] \subseteq \O$ or $r_{\lh(a)+1}[a,B] \subseteq \O^c$.
\end{axiom}

Note that, by \textbf{A3}, the $B$ in \textbf{A4} may instead be taken to be an element of $[a,A]$. By Axiom \textbf{A1}, we may identify each element $A \in \cal{R}$ with a sequence of elements of $\cal{AR}$, via the map $A \mapsto (r_n(A))_{n<\omega}$. Therefore, we may identify $\cal{R}$ with a subset of $\cal{AR}^\N$.

We shall focus on the following topologies on $\cal{R}$.
\begin{itemize}
    \item The \emph{metric topology} induced by the first difference metric, where for $A,B \in \cal{R}$, $d(A,A) = 0$ and $d(A,B) = \frac{1}{2^n}$ for $A \neq B$, where $n$ is the least integer such that $r_n(A) \neq r_n(B)$. The sets of the form $[a]$ for $a \in \cal{AR}$ form a neighbourhood base of this topology. If $\cal{AR}$ is countable, then $\cal{R}$ is a Polish space under this metric topology.

    \item The \emph{Ellentuck topology} generated by open sets of the form $[a,A]$ for $A \in \cal{R}$ and $a \in \cal{AR}\restrictedto A$.
\end{itemize}
Unless stated otherwise, all topological properties of subsets of $\cal{R}$ mentioned in this paper are with respect to the metric topology.

\begin{definition}
    We say that $(\cal{R},\leq,r)$ is a \emph{closed triple} if $\cal{R}$ is a metrically closed subset of $\cal{AR}^\N$.
\end{definition}

\begin{definition}
    Let $(\cal{R},\leq,r)$ be a closed triple satisfying \textbf{A1} and \textbf{A2}. A subset $\X \subseteq \cal{R}$ is \emph{Ramsey} if for all $A \in \cal{R}$ and $a \in \cal{AR}\restrictedto A$, there exists some $B \in [a,A]$ such that $[a,B] \subseteq \X$ or $[a,B] \subseteq \X^c$. If $[a,B] \subseteq \X^c$ always holds, then $\X$ is \emph{Ramsey null}.
\end{definition}

Todor\v{c}evi\'{c} subsequently showed in \cite{T10} the following theorem:

\begin{theorem}[{\cite[Abstract Ellentuck theorem]{T10}}]
\label{thm:abstract.Ellentuck.thm}
    Let $(\cal{R},\leq,r)$ be a closed triple satisfying \textbf{A1}--\textbf{A4}. Then a subset of $\cal{R}$ is Ramsey iff it has the property of Baire under the Ellentuck topology, and it is Ramsey null iff it is meagre under the Ellentuck topology. 
\end{theorem}

In particular, if $\cal{AR}$ is countable (so the metric topology is Polish), then every analytic set is Ramsey. A closed triple $(\cal{R},\leq,r)$ satisfying \textbf{A1} and \textbf{A2} is a \emph{topological Ramsey space} if it satisfies the conclusion of Theorem \ref{thm:abstract.Ellentuck.thm}. Theorem \ref{thm:abstract.Ellentuck.thm} thus asserts that every closed triple satisfying \textbf{A1}--\textbf{A4} is a topological Ramsey space.

We shall prove a lemma asserting that the axioms \textbf{A1}--\textbf{A3} allow us to perform diagonalisation over a family of dense open sets. The following definition is a variant of \cite[Definition 3.6]{DMN15}, and was introduced earlier by the author in \cite{Y24,Y25}.

\begin{definition}
\label{def:dense.open.diagonalisation}
    Let $A \in \cal{R}$ and $a \in \cal{AR}\restrictedto A$.
    \begin{itemize}
        \item A family of subsets $\vec{\D} = \{\D_b\}_{b \in \cal{AR}\restrictedto[a,A]}$ is \emph{dense open below $[a,A]$} if for all $b \in \cal{AR}\restrictedto[a,A]$, $\D_b$ is a $\leq$-downward closed subset of $[b,A]$, and for all $B \in [b,A]$, there exists some $C \in [b,B]$ such that $C \in \D_b$.
        
        \item Let $\vec{\D} = \{\D_b\}_{b \in \cal{AR}\restrictedto[a,A]}$ be dense open below $[a,A]$. We say that $B \in [a,A]$ \emph{diagonalises} $\vec{\D}$ if for all $b \in \cal{AR}\restrictedto[a,B]$, there exists some $A_b \in \D_b$ such that $[b,B] \subseteq [b,A_b]$.
    \end{itemize}
\end{definition}

\begin{lemma}
\label{lem:R.is.semiselective.improved}
    Let $(\cal{R},\leq,r)$ be a closed triple satisfying \textbf{A1}--\textbf{A3}, and let $A \in \cal{R}$. Let $\vec{\D} = \{\D_b\}_{b\in\cal{AR}\restrictedto A}$ be a family of subsets that is dense open below $[\emptyset,A]$. Then for all $a \in \cal{AR}\restrictedto A$, there exists some $B \in [a,A]$ such that for all $b \in \cal{AR}\restrictedto B$ with $\depth_B(b) \geq \lh(a)$, there exists some $A_b \in \D_b$ such that $[b,B] \subseteq [b,A_b]$.
\end{lemma}

\begin{proof}
    Fix some $A \in \cal{R}$ and $a \in \cal{AR}\restrictedto A$, and suppose that $\vec{\D} = \{\D_b\}_{b\in\cal{AR}\restrictedto A}$ is dense open below $[\emptyset,A]$. We shall define a fusion sequence $(A_n)_{n<\omega}$ in $[a,A]$, with $a_n = r_{\lh(a)+n}(A_n)$, such that $A_{n+1} \in [a_n,A_n]$: By \textbf{A3}(1), let $A_0 \in [a,A]$, and suppose that $A_n$ has been defined. Let $\{b_n^i : i < N\}$ enumerate the set of all $b \in \cal{AR}\restrictedto A_n$ such that $\depth_{A_n}(b) = \lh(a) + n$, which is finite by \textbf{A2}(1). We define another sequence $(A_n^i)_{i \leq N}$ in $[a_n,A_n]$ as follows: Let $A_n^0 := A_n$. Suppose that $A_n^i \in [a_n,A_n]$ has been defined. Then $r_j(A_n^i) = r_j(A_n)$ for all $j \leq \lh(a)+n$, so $\depth_{A_n^i}(b_n^i) = \depth_{A_n}(b_n^i) = \lh(a)+n$, and by \textbf{A3}(1) we may let $B_n^i \in [b_n^i,A_n^i]$. Since $\D_{b_n^i}$ is dense open in $[b_n^i,A]$, there exists some $A_{b_n^i} \in \D_{b_n^i}$ and $C_n^i \in [b_n^i,B_n^i]$ such that $[b_n^i,C_n^i] \subseteq [b_n^i,A_{b_n^i}]$. By \textbf{A3}(2), we may let $A_n^{i+1} \in [\depth_{A_n^i}(b_n^i),A_n^i] = [a_n,A_n^i]$ be such that $[b_n^i,A_n^{i+1}] \subseteq [b_n^i,C_n^i]$. We complete the induction by letting $A_{n+1} := A_n^N$. Since $\cal{R}$ is closed, $B := \lim_{n\to\infty} a_n$ is an element of $\cal{R}$, and $B \in [a,A]$ and $B \leq A_n^i$ for all $n$ and $i \leq N$ by \textbf{A2}(2). We shall show that $B$ is the required element of $[a,A]$.

    Let $b \in \cal{AR}\restrictedto B$ be such that $\depth_B(b) = \lh(a) + n$ for some $n \geq 0$. Since $r_j(B) = r_j(A_n)$ for all $j \leq \lh(a)+n$, we have that $\depth_{A_n}(b) = \lh(a)+n$, so $b = b_n^i$ for some $i$ in the induction step of the proof. Then $[b,B] \subseteq [b,A_n^{i+1}] \subseteq [b,C_n^i] \subseteq [b,A_b]$ and $A_b \in \D_b$, as desired.
\end{proof}

In particular (as $\depth_B(b) \geq \depth_B(a) = \lh(a)$ for all $b \in \cal{AR}\restrictedto[a,B]$), we have the following:

\begin{lemma}[{\cite[Lemma 3.16]{Y24}}]
\label{lem:R.is.semiselective}
    Let $(\cal{R},\leq,r)$ be a closed triple satisfying \textbf{A1}--\textbf{A3}, and let $A \in \cal{R}$ and $a \in \cal{AR}\restrictedto A$. Then every family of subsets $\vec{\D} = \{\D_b\}_{b \in \cal{AR}\restrictedto[a,A]}$ that is dense open below $[a,A]$ has a diagonalisation.
\end{lemma}

\section{Abstract Mathias forcing}
\label{sec:abstract.Mathias.forcing}
We recap the abstract Mathias forcing introduced in \cite{DMN15}, and prove that it has the Mathias and Prikry properties. Most of the content in this section is adapted from \cite[\S6]{Y25}.

Fix a closed triple $(\cal{R},\leq,r)$ satisfying \textbf{A1}--\textbf{A4}. The abstract Mathias forcing is defined as the forcing poset $(\mathbb{M}_\cal{R},\leq)$, where:
\begin{align*}
    \mathbb{M}_\cal{R} = \{(a,A) : A \in \cal{R} \text{ and } a \in \cal{AR}\restrictedto A\},
\end{align*}
and that $(b,B) \leq (a,A)$ iff $[b,B] \subseteq [a,A]$. We also write $(b,B) \leq_0 (a,A)$ if $(b,B) \leq (a,A)$ and $b = a$.

\begin{definition}
    An element $A \in \cal{R}$ is said to be $\mathbb{M}_\cal{R}$-generic (over a ground model $\V$) if:
    \begin{align*}
        G_A := \{(a,B) \in \mathbb{M}_\cal{R} : A \in [a,B]\}
    \end{align*}
    is a generic filter of $\mathbb{M}_\cal{R}$ (over $\V$).
\end{definition}

\begin{lemma}
\label{lem:generic.filter.and.element}
\hfill
    \begin{enumerate}
        \item If $G$ is a generic filter of $\mathbb{M}_\cal{R}$, then there exists a unique $\mathbb{M}_\cal{R}$-generic $A_G \in \cal{R}$ such that $G = \{(a,A) \in \mathbb{M}_\cal{R} : A_G \in [a,A]\}$.

        \item If $A \in \cal{R}$ is $\mathbb{M}_\cal{R}$-generic, then $A_{G_A} = A$.

        \item If $G$ is a generic filter of $\mathbb{M}_\cal{R}$, then $G_{A_G} = G$.
    \end{enumerate}
\end{lemma}

\begin{proof}
    We only prove (1) --- (2) follows from the uniqueness of $A_G$ in (1), and (3) follows easily from the definitions. Let $\D_n := \{(a,A) \in \mathbb{M}_\cal{R} : \lh(a) \geq n\}$. We see that $\D_n$ is dense open in $\mathbb{M}_\cal{R}$ for all $n$ --- if $(a,A) \notin \D_n$ (i.e. $\lh(a) < n$) and $b \in r_n[a,A]$, then by \textbf{A3}(1) we have that $[b,A] \neq \emptyset$. Let $B \in [b,A]$, and we have that $[b,B] \subseteq [a,A]$, so $(b,B) \leq (a,A)$ and $(b,B) \in \D_n$. Therefore, $G \cap \D_n \neq \emptyset$ for all $n$, and since $G$ is a filter, for all $(a,A),(b,B) \in G$, $a \sqsubseteq b$ or $b \sqsubseteq a$. Thus, there exists a unique $A_G \in \cal{R}$ such that for all $n$, $(r_n(A_G),A) \in G$ for some $A$. It's easy to verify that $G = G_{A_G}$.
\end{proof} 

\subsection{Abstract combinatorial forcing}
We recap the abstract combinatorial forcing developed in \cite[\S4.3]{T10}. We fix a subset $\X \subseteq \cal{R}$.

\begin{definition}
    Let $A \in \cal{R}$ and let $a \in \cal{AR}$. We say that:
    \begin{enumerate}
        \item $A$ \emph{accepts} $a$ if $[a,A] \subseteq \X$.

        \item $A$ \emph{rejects} $a$ if $a \in \cal{AR}\restrictedto A$, and for all $B \in [a,A]$, $B$ does not accept $a$.

        \item $A$ \emph{decides} $a$ if $A$ accepts or rejects $a$.
    \end{enumerate}
\end{definition}

\begin{lemma}[{\cite[Lemma 4.31]{T10}}]
\label{lem:combinatorial.forcing.basics}
    Let $A \in \cal{R}$ and let $a \in \cal{AR}$.
    \begin{enumerate}
        \item If $a \notin\cal{AR}\restrictedto A$, $A$ accepts $a$ (as $[a,A] = \emptyset$).

        \item If $A$ accepts $a$ and $B \leq A$, then $B$ accepts $a$.

        \item If $A$ rejects $a$, $B \leq A$ and $a \in \cal{AR}\restrictedto B$, then $B$ rejects $a$.

        \item If $A$ decides $a$ and $B \leq A$, then $B$ decides $a$.

        \item If $a \in \cal{AR}\restrictedto A$, then there exists some $B \in [\depth_A(a),A]$ that decides $a$.

        \item If $a \in \cal{AR}\restrictedto A$ and $A$ decides $a$, then for all $B \in [\depth_A(a),A]$, $B$ decides $a$ in the same way $A$ does.

        \item $A$ accepts $a$ iff $A$ accepts every $b \in r_{\lh(a)+1}[a,A]$.
    \end{enumerate}
\end{lemma}

\begin{lemma}[{\cite[Lemma 4.35]{T10}}]
\label{lem:completely.reject.reduction}
    If $A$ rejects $a$, then there exists some $B \in [a,A]$ which rejects every $b \in \cal{AR}\restrictedto[a,B]$.
\end{lemma}

We refer readers to \cite[\S4]{T10} for the proofs.

\subsection{Mathias and Prikry properties}
We may now introduce the Mathias and Prikry properties.

\begin{definition}
    Let $(\cal{R},\leq,r)$ be a closed triple satisfying \textbf{A1}--\textbf{A4}.
    \begin{enumerate}
        \item $\mathbb{M}_\cal{R}$ satisfies the \emph{Mathias property} if whenever $A \in \cal{R}$ is $\mathbb{M}_\cal{R}$-generic over a ground model $\V$ and $B \leq A$, $B$ is also $\mathbb{M}_\cal{R}$-generic over $\V$.

        \item $\mathbb{M}_\cal{R}$ satisfies the \emph{Prikry property} if whenever $\varphi$ is a sentence in the forcing language, then we have that for all $p \in \mathbb{M}_\cal{R}$, there exists some $q \leq_0 p$ such that $q \forces \varphi$ or $q \forces \neg\varphi$.
    \end{enumerate}
\end{definition}

Given a subset $\D \subseteq \mathbb{M}_\cal{R}$, we let:
\begin{align*}
    \bar{\D} &:= \bigcup_{(a,A) \in \D} [a,A].
\end{align*}
Note that $\bar{\D}$ is an Ellentuck open subset of $\cal{R}$ in $\V$. Thus, if $\D$ is dense open in $\mathbb{M}_\cal{R}$, then for all $A \in \cal{R}$ and $a \in \cal{AR}\restrictedto A$, $A$ does not reject $a$ w.r.t. $\bar{\D}$ (otherwise, by Lemma \ref{lem:completely.reject.reduction} there exists some $B \in [a,A]$ that rejects every $b \in \cal{AR}\restrictedto[a,B]$, contradicting that $\D$ is dense in $\mathbb{M}_\cal{R}$). In other words, for all $A \in \cal{R}$ and $a \in \cal{AR}\restrictedto A$, there exists some $B \in [a,A]$ such that $[a,B] \subseteq \bar{\D}$.

\begin{definition}
\label{def:capture}
    Let $\D \subseteq \mathbb{M}_\cal{R}$. We say that $(a,A) \in \mathbb{M}_\cal{R}$ \emph{captures} $\D$ if for all $B \in [a,A]$, there exists some initial segment $b \in \cal{AR}\restrictedto[a,B]$ of $B$ such that $(b,A) \in \D$.
\end{definition}

We remark that the statement ``$(a,A)$ captures $\D$'' is absolute between transitive models: $(a,A)$ captures $\D$ iff the tree $\{b \in \cal{AR}\restrictedto[a,A] : (c,A) \notin \D \text{ for all } c \sqsubseteq b \text{ with } a \sqsubseteq c\}$, ordered by $\sqsubseteq$, has no infinite branch, and well-foundedness is absolute between transitive models.

\begin{lemma}
\label{lem:capturing.possible}
    Let $\D \subseteq \mathbb{M}_\cal{R}$ be dense open. Then for all $A \in \cal{R}$ and $a \in \cal{AR}\restrictedto A$, there exists some $B \in [a,A]$ such that $(a,B)$ captures $\D$.
\end{lemma}

\begin{proof}
    By the observation above, we may let $A' \in [a,A]$ be such that $[a,A'] \subseteq \bar{\D}$. For each $b \in \cal{AR}\restrictedto[a,A']$, let:
    \begin{align*}
        \D_b := \{C \in [b,A'] : (b,C) \in \D \text{ or } \forall C' \in [b,C] \, (b,C') \notin \D\}.
    \end{align*}
    It's clear that $\{\D_b\}_{b\in\cal{AR}\restrictedto[a,A']}$ is dense open below $[a,A']$. By Lemma \ref{lem:R.is.semiselective}, there exists some $B \in [a,A']$ that diagonalises $\{\D_b\}_{b\in\cal{AR}\restrictedto[a,A']}$. We shall show that $(a,B)$ captures $\D$.

    Let $D \in [a,B]$. Since $D \in [a,A'] \subseteq \bar{\D}$, there exists some $(b,C) \in \D$ such that $D \in [b,C]$, and we may assume that $a \sqsubseteq b$. As $B$ diagonalises $\{\D_b\}_{b\in\cal{AR}\restrictedto[a,A']}$ and $b \in \cal{AR}\restrictedto[a,B]$, there exists some $A_b \in \D_b$ such that $[b,B] \subseteq [b,A_b]$, so either $(b,A_b) \in \D$ or $(b,B') \notin \D$ for all $B' \in [b,A_b]$. But the latter case is impossible as $D \in [b,B] \subseteq [b,A_b]$ and $(b,D) \in \D$ (as $(b,C) \in \D$ and $\D$ is open). Hence $(b,A_b) \in \D$, and as $(b,B) \leq (b,A_b)$, $(b,B) \in \D$. Since $D$ is arbitrary, $(a,B)$ captures $\D$.
\end{proof}

\begin{proposition}
\label{prop:semiselective.implies.Mathias}
    $\mathbb{M}_\cal{R}$ has the Mathias property.
\end{proposition}

\begin{proof}
    Let $A_G$ be an $\mathbb{M}_\cal{R}$-generic element of $\cal{R}$ (over some ground model $\V$), and let $D \leq A_G$. We need to show that the filter:
    \begin{align*}
        H := \{(a,A) \in \mathbb{M}_\cal{R} : D \in [a,A]\}
    \end{align*}
    is generic. Let $\D \subseteq \mathbb{M}_\cal{R}$ be a dense open subset (in the ground model $\V$). For each $a \in \cal{AR}$, we let:
    \begin{align*}
        \E_a &:= \{B \in \cal{R} : a \in \cal{AR}\restrictedto B \text{ and $(a,B)$ captures $\D$}\}.
    \end{align*}
    By Lemma \ref{lem:capturing.possible}, $\E_a \cap [a,A]$ is dense open in $[a,A]$ for all $A \in \cal{R}$ and $a \in \cal{AR}\restrictedto A$. Applying Lemma \ref{lem:R.is.semiselective.improved} to the family $\{\E_b \cap [b,A]\}_{b \in \cal{AR}\restrictedto A}$, we have that for all $(a,A) \in \mathbb{M}_\cal{R}$, there exists some $B \in [a,A]$ such that for all $b \in \cal{AR}\restrictedto B$ with $\depth_B(b) \geq \lh(a)$, there exists some $A_b \in \E_b$ such that $[b,B] \subseteq [b,A_b]$. That is, the set:
    \begin{align*}
        \D' := \{(a,A) \in \mathbb{M}_\cal{R} : {}&\forall b \in \cal{AR}\restrictedto A \, (\depth_A(b) \geq \lh(a) \\
        &\to \exists A_b \in \E_b \, [b,A] \subseteq [b,A_b])\}
    \end{align*}
    is dense in $\mathbb{M}_\cal{R}$. Since $G_{A_G}$ is generic, we may let $(a,A) \in G_{A_G} \cap \D'$, so $A_G \in [a,A]$ and $D \leq A$. Let $n$ be such that $\depth_A(r_n(D)) \geq \lh(a)$, let $d := r_n(D)$, and let $A_d \in \E_d$ be such that $[d,A] \subseteq [d,A_d]$. Then $D \in [d,A] \subseteq [d,A_d]$, and $(d,A_d)$ captures $\D$, so by the absoluteness of capturing (see the remark after Definition \ref{def:capture}) there exists some initial segment $e \in \cal{AR}\restrictedto[d,D]$ of $D$ such that $(e,A_d) \in \D$. Then $D \in [e,A_d]$, so $(e,A_d) \in H \cap \D$. Since $\D$ is arbitrary, $H$ is a generic filter, as required.
\end{proof}

\begin{proposition}
\label{prop:semiselective.implies.Prikry}
    $\mathbb{M}_\cal{R}$ has the Prikry property.
\end{proposition}

\begin{proof}
    Let $\varphi$ be a sentence in the forcing language. Let:
    \begin{align*}
        \D &:= \{q \in \mathbb{M}_\cal{R} : q \forces \varphi\}.
    \end{align*}
    Note that $\D$ is open. We consider doing combinatorial forcing w.r.t. the set $\bar{\D}$. Let $(a,A) \in \mathbb{M}_\cal{R}$. By Lemma \ref{lem:combinatorial.forcing.basics}(5--6), there exists some $B \in [a,A]$ that decides $a$.
    \begin{itemize}
        \item Suppose that $B$ accepts $a$, so $[a,B] \subseteq \bar{\D}$. We shall show that no $(b,C) \leq (a,B)$ forces $\neg\varphi$, so that $(a,B) \forces \varphi$. Otherwise, there exists some $(b,C) \leq (a,B)$ such that $(b,C) \forces \neg\varphi$, and we may assume that $b \sqsubseteq C$. Since $C \in [a,B] \subseteq \bar{\D}$, $C \in [a',B']$ for some $(a',B') \in \mathbb{M}_\cal{R}$ such that $(a',B') \forces \varphi$. Let $n \geq \max\{\lh(b),\lh(a')\}$, and we have that $(r_n(C),C) \leq (b,C)$ and $(r_n(C),C) \leq (a',B')$, so $(r_n(C),C)$ forces both $\varphi$ and $\neg\varphi$, a contradiction. Thus $(a,B) \leq_0 (a,A)$ and $(a,B) \forces \varphi$.

        \item Suppose that $B$ rejects $a$. By Lemma \ref{lem:completely.reject.reduction}, replacing $B$ by an element of $[a,B]$, we may assume that $B$ rejects every $b \in \cal{AR}\restrictedto[a,B]$. We shall show that no $(b,C) \leq (a,B)$ forces $\varphi$, so that $(a,B) \forces \neg\varphi$. Otherwise, $(b,C) \in \D$ for some $(b,C) \leq (a,B)$, and we may assume that $a \sqsubseteq b \sqsubseteq C$. Then $[b,C] \subseteq \bar{\D}$, i.e. $C$ accepts $b$. On the other hand, $b \in \cal{AR}\restrictedto[a,B]$ and $C \leq B$, so $C$ rejects $b$ by Lemma \ref{lem:combinatorial.forcing.basics}(3), which is impossible. Thus $(a,B) \leq_0 (a,A)$ and $(a,B) \forces \neg\varphi$.
    \end{itemize}
\end{proof}

\section{Proof of Theorem \ref{thm:main}}
\label{sec:proof.main}
Throughout this section, we fix a closed triple $(\cal{R},\leq,r)$ satisfying \textbf{A1}--\textbf{A4} such that $\cal{AR}$ is countable. Since $\cal{AR}$ is countable, we may assume that $\cal{AR} \subseteq \V_\omega$, so that $\cal{R} \subseteq \cal{AR}^\N$ is a set of reals, and $\lh$, $\sqsubseteq$ and $\leq_\fin$ are subsets of $\V_\omega$ (hence are coded by reals). 

For a transitive model $M$ of set theory containing $\cal{AR}$, $\lh$, $\sqsubseteq$ and $\leq_\fin$, we write $\cal{R}^M$ for $\cal{R}$ as computed in $M$. Note that closedness and the axioms \textbf{A1}--\textbf{A4} are absolute between such models with the same ordinals. This follows from Shoenfield absoluteness, as closedness and \textbf{A1}--\textbf{A4} are $\Pi_2^1$ statements (except for \textbf{A3}(2), which is $\Pi_3^1$ as written, but one can show that for $a \in \cal{AR}\restrictedto B'$, $[a,B'] \subseteq [a,B]$ iff $\cal{AR}\restrictedto[a,B'] \subseteq \cal{AR}\restrictedto B$, which is arithmetical).

The following lemma is due to Solovay, and a proof may be found in \cite{Sol70}.

\begin{lemma}[Solovay]
\label{lem:Solovay}
    Let $\kappa$ be an inaccessible cardinal, let $\P \in \V$ be a forcing poset of size $<\kappa$, and let $G$ be $\Coll(\omega,{<}\kappa)$-generic over $\V$.
    \begin{enumerate}
        \item For every $p \in \P$, there exists some $H \in \V[G]$ that is $\P$-generic over $\V$ such that $p \in H$.

        \item If $H \in \V[G]$ is $\P$-generic over $\V$, then there exists some $G^*$ that is $\Coll(\omega,{<}\kappa)$-generic over $\V[H]$ such that $\V[H][G^*] = \V[G]$.
    \end{enumerate}
\end{lemma}

We are now ready to prove Theorem \ref{thm:main}. The proof follows that of \cite[Theorem 1.7]{Y25}.

\begin{proof}[Proof of Theorem \ref{thm:main}]
    Let $\X \in \L(\R)^{\V[G]}$ be a subset of $\cal{R}$. Note that $\cal{R}^{\V[G]} = \cal{R}^{\L(\R)^{\V[G]}}$, as every element of $\cal{AR}^\N \cap \V[G]$ is coded by a real, so it suffices to show that $\X$ is Ramsey in $\V[G]$. Fix some $A \in \cal{R}^{\V[G]}$ and $a \in \cal{AR}\restrictedto A$. We need to show that there exists some $B \in [a,A]$ such that $[a,B] \subseteq \X$ or $[a,B] \subseteq \X^c$.
    
    Since every set in $\L(\R)^{\V[G]}$ is ordinal-definable from some real in $\V[G]$, there exist a formula $\phi$, some real $r \in \V[G]$ and a finite sequence of ordinals $\alpha$ such that for all $B \in \cal{R}^{\V[G]}$:
    \begin{align*}
        B \in \X \iff \V[G] \models \phi[B,r,\alpha].
    \end{align*}
    Let $\phi^*(x,y,w)$ be the formula asserting that $\one \forces_{\Coll(\omega,{<}\kappa)} \phi(\check{x},\check{y},\check{w})$. By the $\kappa$-chain condition of $\Coll(\omega,{<}\kappa)$, there exists some $\lambda < \kappa$ such that $r,A \in \V[G\restrictedto\lambda]$. Let $\W := \V[G\restrictedto\lambda]$, and note that $\kappa$ remains inaccessible in $\W$ and $\V[G] = \W[G\restrictedto[\lambda,\kappa)]$. We consider, in $\W$, the abstract Mathias forcing $\mathbb{M} := (\mathbb{M}_\cal{R})^\W$, and note that $(a,A) \in \mathbb{M}$. For $B \in \cal{R}^{\V[G]}$, recall from the beginning of \S\ref{sec:abstract.Mathias.forcing} that $G_B := \{(c,C) \in \mathbb{M} : B \in [c,C]\}$, so that $B$ is $\mathbb{M}$-generic over $\W$ iff $G_B$ is an $\mathbb{M}$-generic filter over $\W$, in which case $B = A_{G_B}$ and $\W[G_B] = \W[B]$ (Lemma \ref{lem:generic.filter.and.element}).

    \begin{claim}
    \label{claim:generic.in.X}
        Let $B \in \cal{R}^{\V[G]}$ be $\mathbb{M}$-generic over $\W$. Then:
        \begin{align*}
            B \in \X \iff \W[B] \models \phi^*[B,r,\alpha].
        \end{align*}
    \end{claim}

    \begin{midproof}
        Since $\W \models |\mathbb{M}| \leq 2^{\aleph_0} < \kappa$ and $G_B \in \V[G]$ is $\mathbb{M}$-generic over $\W$, by Lemma \ref{lem:Solovay}(2) applied in $\W$ there exists some $\Coll(\omega,{<}\kappa)$-generic $G^*$ over $\W[B]$ such that $\W[B][G^*] = \V[G]$. Therefore:
        \begin{align*}
            B \in \X &\iff \V[G] \models \phi[B,r,\alpha] \\
            &\iff \W[B][G^*] \models \phi[B,r,\alpha] \\
            &\iff \W[B] \models \phi^*[B,r,\alpha],
        \end{align*}
        where the last equivalence follows from the homogeneity of $\Coll(\omega,{<}\kappa)$.
    \end{midproof}

    Let $\dot{A}$ be the canonical $\mathbb{M}$-name for the $\mathbb{M}$-generic element of $\cal{R}$. By the Prikry property of $\mathbb{M}$ (Proposition \ref{prop:semiselective.implies.Prikry}), there exists some $(a,A') \leq_0 (a,A)$ in $\mathbb{M}$ which decides the sentence $\phi^*(\dot{A},\check{r},\check{\alpha})$. By Lemma \ref{lem:Solovay}(1) applied in $\W$, there exists some $H \in \V[G]$ that is $\mathbb{M}$-generic over $\W$ such that $(a,A') \in H$. Let $A_H \in \cal{R}^{\V[G]}$ be the $\mathbb{M}$-generic element of $\cal{R}$ such that $H = G_{A_H}$ (which exists by Lemma \ref{lem:generic.filter.and.element}), and note that $A_H \in [a,A'] \subseteq [a,A]$. We shall finish the proof by showing that $[a,A_H] \subseteq \X$ or $[a,A_H] \subseteq \X^c$. 

    Suppose that $(a,A') \forces \phi^*(\dot{A},\check{r},\check{\alpha})$, and let $B \in [a,A_H]$. By the Mathias property of $\mathbb{M}$ (Proposition \ref{prop:semiselective.implies.Mathias}), $B$ is also $\mathbb{M}$-generic over $\W$. Since $B \leq A_H \leq A'$, we have $B \in [a,A']$, so $(a,A') \in G_B$, and as $\dot{A}^{G_B} = B$, $\W[B] \models \phi^*[B,r,\alpha]$. Hence $B \in \X$ by Claim \ref{claim:generic.in.X}. Therefore, $[a,A_H] \subseteq \X$. Alternatively, if $(a,A') \forces \neg\phi^*(\dot{A},\check{r},\check{\alpha})$, then $[a,A_H] \subseteq \X^c$.
\end{proof}

We remark that Theorem \ref{thm:main} appeared as \cite[Theorem 6.1]{Y24}, with the following justification: Since every subset of a Polish space has the property of Baire in the Solovay model and the Ellentuck topology refines the Polish topology, every subset of $\cal{R}$ is Ramsey by the abstract Ellentuck theorem. This is incorrect, as a set with the property of Baire with respect to a topology need not have the property of Baire with respect to a finer topology. In fact, the implication ``every metrically Baire subset of $[\N]^\infty$ is Ramsey'' fails in $\ZFC$: Let $A \in [\N]^\infty$ be coinfinite, and let $\X \subseteq [A]^\infty$ be any non-Ramsey set. Since $[A]^\infty$ is (metrically) nowhere dense, so is $\X$, so $\X$ is metrically Baire but not Ramsey.

\printbibliography[title={References}]

\end{document}